\documentclass[a4paper,USenglish,cleveref, autoref, thm-restate]{lipics-v2021}

\hideLIPIcs  

\usepackage{float}
\usepackage{breqn}
\usepackage{mathtools}
\usepackage{amssymb}
\usepackage{bm}
\usepackage[sort]{cite}
\usepackage{wasysym}

\DeclareMathOperator{\tw}{tw}
\DeclareMathOperator{\rtw}{rtw}
\DeclareMathOperator{\rDist}{rDist}
\DeclareMathOperator{\rad}{rad}
\newcommand{\dist}{d}

\newcommand{\area}{A}

\renewcommand{\epsilon}{\varepsilon}
\renewcommand{\P}{\ensuremath{\text{\Large \pentagon}}}

\DeclarePairedDelimiter{\norm}{\lVert}{\rVert}

\newcolumntype{P}[1]{>{\centering\arraybackslash}p{#1}}

\graphicspath{{./figures/}}

\title{Intersection Graphs with and without Product Structure} 

\author{Laura Merker}{Karlsruhe Institute of Technology (KIT), Germany}{laura.merker2@kit.edu}{https://orcid.org/0000-0003-1961-4531}{}
\author{Lena Scherzer}{Karlsruhe Institute of Technology (KIT), Germany}{lena.scherzer@student.kit.edu}{}{supported by the Deutsche Forschungsgemeinschaft - 520723789} 
\author{Samuel Schneider}{Karlsruhe Institute of Technology (KIT), Germany}{samuel.schneider@student.kit.edu}{https://orcid.org/0009-0002-9680-4048}{supported by the Deutsche Forschungsgemeinschaft - 520723789} 
\author{Torsten Ueckerdt}{Karlsruhe Institute of Technology (KIT), Germany}{torsten.ueckerdt@kit.edu}{https://orcid.org/0000-0002-0645-9715}{supported by the Deutsche Forschungsgemeinschaft - 520723789}

\authorrunning{L. Merker, L. Scherzer, S. Schneider, and T. Ueckerdt} 

\Copyright{Laura Merker, Lena Scherzer, Samuel Schneider, and Torsten Ueckerdt} 



\acknowledgements{}

\nolinenumbers 

\EventEditors{John Q. Open and Joan R. Access}
\EventNoEds{2}
\EventLongTitle{42nd Conference on Very Important Topics (CVIT 2016)}
\EventShortTitle{CVIT 2016}
\EventAcronym{CVIT}
\EventYear{2016}
\EventDate{December 24--27, 2016}
\EventLocation{Little Whinging, United Kingdom}
\EventLogo{}
\SeriesVolume{42}
\ArticleNo{23}

\begin{document}

\maketitle

\begin{abstract}
    A graph class $\mathcal{G}$ admits \emph{product structure} if there exists a constant $k$ such that every $G \in \mathcal{G}$ is a subgraph of $H \boxtimes P$ for a path $P$ and some graph $H$ of treewidth~$k$.
    Famously, the class of planar graphs, as well as many beyond-planar graph classes are known to admit product structure.
    However, we have only few tools to prove the absence of product structure, and hence know of only a few interesting examples of classes.
    Motivated by the transition between product structure and no product structure, we investigate subclasses of intersection graphs in the plane (e.g., disk intersection graphs) and present necessary and sufficient conditions for these to admit product structure.

    Specifically, for a set $S \subset \mathbb{R}^2$ (e.g., a disk) and a real number $\alpha \in [0,1]$, we consider \emph{intersection graphs of $\alpha$-free homothetic copies of $S$}.
    That is, each vertex $v$ is a homothetic copy of $S$ of which at least an $\alpha$-portion is not covered by other vertices, and there is an edge between $u$ and $v$ if and only if $u \cap v \neq \emptyset$.
    For $\alpha = 1$ we have contact graphs, which are in most cases planar, and hence admit product structure.
    For $\alpha = 0$ we have (among others) all complete graphs, and hence no product structure.
    In general, there is a threshold value $\alpha^*(S) \in [0,1]$ such that $\alpha$-free homothetic copies of $S$ admit product structure for all $\alpha > \alpha^*(S)$ and do not admit product structure for all $\alpha < \alpha^*(S)$.
    
    We show for a large family of sets $S$, including all triangles and all trapezoids, that it holds $\alpha^*(S) = 1$, i.e., we have no product structure, except for the contact graphs (when $\alpha= 1$).
    For other sets $S$, including regular $n$-gons for infinitely many values of $n$, we show that $0 < \alpha^*(S) < 1$ by proving upper and lower bounds.
\end{abstract}

\clearpage

\section{Introduction}
\label{sec:introduction}

We are interested in properties $P$ of graph classes, for which, if a graph class $\mathcal{G}$ admits $P$, this certifies that $\mathcal{G}$ is well-behaving in some sense. 
For example, having property $P$ for $\mathcal{G}$ can provide a common structure of graphs $G \in \mathcal{G}$, which can be exploited to prove statements for all graphs in $\mathcal{G}$, or to derive efficient algorithms for graphs in $\mathcal{G}$.
A particularly nice 
property is that of having \emph{bounded treewidth}, i.e., that there exists a constant $t$ such that every graph $G \in \mathcal{G}$ is a subgraph of $H$ for some $t$-tree\footnote{Definitions of $t$-trees, treewidth, and strong products will be given in \cref{sec:preliminaries}.} $H$.
Using the simple structure of $t$-trees, one can for example show that graphs in $\mathcal{G}$ have small balanced separators, or that \textsc{MaxIndependentSet} can be solved in linear time.
However, the $n \times n$-grid graph has treewidth~$n$ and hence, already the class of planar graphs does not have bounded treewidth.

In 2019, Dujmovi\'{c} et al.\ \cite{PlanarGraphsQueueNumber} introduced with \textbf{product structure} a novel concept that generalizes the property of having bounded treewidth.
We say that a graph class $\mathcal{G}$ admits \emph{product structure} if there exists a constant $t$ such that every graph $G \in \mathcal{G}$ is a subgraph of~$H \boxtimes P$ for a path $P$ and some $t$-tree $H$.
For example, the $n \times n$-grid graph is a subgraph of~$P_n \boxtimes P_n$, i.e., contained in the strong product of a path and a graph of treewidth~$1$.
In fact, the class of all planar graphs admits product structure with constant $t = 6$~\cite{ImprovedPlanarGraphProductStructureTheorem}.

This constant $t$ is also called the \emph{row treewidth} (denoted $\rtw$) and specifically, for a graph~$G$ we write $\rtw(G) \leq t$ if $G \subseteq P \boxtimes H$ for a path $P$ and some $t$-tree $H$.
The vertices of~$P \boxtimes H$ are partitioned into ``rows'' (one row for each vertex of $P$), with each row inducing a copy of $H$.
In particular, from any vertex-ordering $\sigma$ of $H$, we obtain a natural drawing 
that reflects the structure of the graph by putting the vertices of the $i$-th row on $y$-coordinate (roughly)~$i$ and $x$-coordinate according to $\sigma$.
See \cref{fig:unit-disk-graph} for some illustrating examples.

\begin{figure}[ht]
    \centering
    \includegraphics{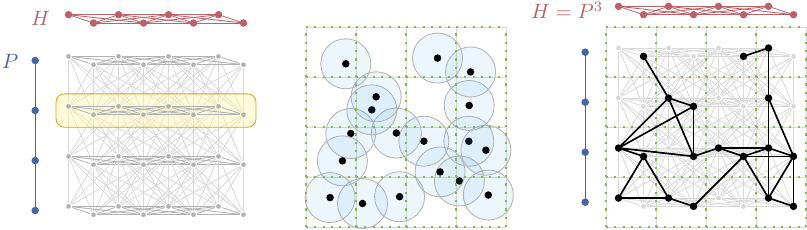}
    \caption{
        The strong product $P \boxtimes H$ with a row highlighted (left),
        a unit disk intersection representation of a graph $G$ (center),
        and its product structure representation $G \subseteq P \boxtimes H$ (right).
        }
    \label{fig:unit-disk-graph}
\end{figure}

Interestingly, one can reverse this procedure in case~$G$ is an intersection graph of unit disks in~$\mathbb{R}^2$ as follows.
Roughly speaking, we superimpose a square grid on the intersection representation and let~$w$ be the largest number of disk centers in any grid cell.
Then we find~$G$ as a subgraph of~$P \boxtimes H$, where~$H = P^{2w-1}$ is the $(2w-1)$-th power of a path, i.e., a~$t$-tree for $t = 2w-1$.
This way, one can conclude for every $w$ that $K_{w/4}$-free unit disk intersection graphs admit product structure.
We refer to~\cite{UnitDisks} for a complete, formal proof.

\subparagraph*{Linear local treewidth.}

On the other side, we can show that a graph class $\mathcal{G}$ has \emph{no} product structure by showing that $\mathcal{G}$ fails to have another (easier to check) property that in turn would be necessary for product structure.
To this end, note that in the product $P \boxtimes H$ each edge either runs within a row or between two consecutive rows. 
It follows that for every $k$ the \emph{$k$-th closed neighborhood} $N^k[v] = \{u \in V \mid \dist(u,v) \leq k\}$ of a vertex $v$ in $P \boxtimes H$ is completely contained in at most~$2k+1$ rows, each having treewidth~$t = \tw(H)$.
Therefore, the treewidth of $N^k[v]$ is at most~$(2k+1)t$, i.e., grows only linearly in~$k$ for constant row treewidth~$t$.

\begin{definition}
    \label{def:linear-local-treewidth}
    %
    A graph class $\mathcal{G}$ has \emph{linear local treewidth} if for all graphs $G \in \mathcal{G}$ and all vertices $v \in V(G)$ the treewidth of the $k$-th closed neighborhood of $v$ is in $\mathcal{O}(k)$.
\end{definition}

Linear local treewidth is a necessary (though not sufficient) condition for a class to admit product structure~\cite[Lemma 6]{DUJMOVIC2017111,PlanarGraphsQueueNumber}.
For example, in $K_4$-free disk intersection graphs we can have an $n \times n$-grid in the neighborhood of a single vertex (\cref{fig:disk-graph}-left), and hence $K_4$-free disk intersection graphs do not admit product structure.
Similarly, the $n\times n \times n$-grid graph has treewidth $\Omega(n^2)$, but lies completely in $N^{3n}[v]$ of any of its vertices $v$ (\cref{fig:disk-graph}-center).
Hence, $K_3$-free intersection graphs of unit balls in $\mathbb{R}^3$ do not admit product structure.

\begin{figure}[ht]
    \centering
    \includegraphics{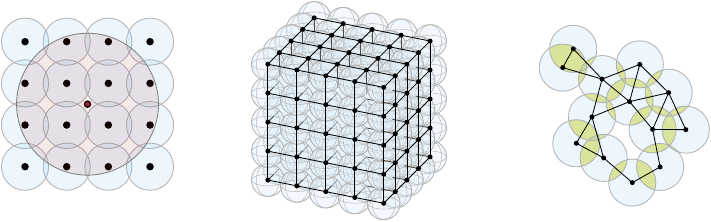}
    \caption{%
        $K_4$-free disk intersection graphs (left) and $K_3$-free unit ball intersection graphs (center) with no product structure, and a $\frac{1}{2}$-free disk intersection graph with the free area in blue (right).
    }
    \label{fig:disk-graph}
\end{figure}



\subparagraph*{Intersection graphs of $\bm{\alpha}$-free homothetic shapes.}

We consider the question when a class of intersection graphs has product structure, and when not.
Crucially, we want the vertices to be represented by homothetic shapes in $\mathbb{R}^2$ of different sizes.
By the discussion above, we must bound the clique size, as well as the size of grids in the neighborhood of any vertex.
We do this by requiring that for every vertex, its corresponding set has an $\alpha$-fraction of its area disjoint from all other shapes (\cref{fig:disk-graph}-right)%
\footnote{This restriction is better suited to investigate the threshold between product structure and no product structure than the (possibly more common) ply, i.e., the number of shapes that may meet at the same point.
In fact, \cref{fig:disk-graph}~left already shows that graphs with ply 3 do not admit product structure.}%
.
For a shape $S \subseteq \mathbb{R}^2$, let us denote its area by $\norm{S}$.

\begin{definition}
    \label{def:alpha-free}
    Let $S \subseteq \mathbb{R}^2$ be a set and $\alpha \in [0,1]$ be a real number.
    A collection~$\mathcal{C} = \{S_v\}_{v \in V}$ of homothetic (obtained from $S$ by positive scaling and translation) copies of $S$ is \emph{$\alpha$-free} if
    \[
        \text{for every $v \in V$ we have } \qquad 
        \norm{S_v - \bigcup_{u \in V-v} S_u} \geq \alpha \cdot \norm{S_v}.
    \]
    We denote by $\mathcal{G}(S,\alpha)$ the class of all intersection graphs of $\alpha$-free homothetic copies of~$S$.
\end{definition}

Whether $\mathcal{G}(S,\alpha)$ has product structure or not clearly depends on $S$ and $\alpha$.
In general, there is a threshold value $\alpha^*(S) \in [0,1]$ such that $\mathcal{G}(S,\alpha)$ has product structure for $\alpha > \alpha^*(S)$, and no product structure for $\alpha < \alpha^*(S)$.
For an integer $n \geq 3$, we denote by $\P_n \subseteq \mathbb{R}^2$ a fixed regular $n$-gon with area~$\norm{\P_n} = 1$.
Our main interest is to determine $\alpha^*(\P_n)$.

When two homothetic copies $S,S'$ of $\P_n$ intersect, a number $m \leq n$ of corners of one, say $S$, is contained in the other, $S'$.
Given that exactly $m$ corners of $S$ are covered by $S'$, the smallest area of $S$ is covered when $S \cap S'$ is the convex hull of $m$ consecutive corners of $S$.
Let us call such a polygon an \emph{$n$-gon segment with $m$ corners} and denote it by $\P_n^m$.
See \cref{fig:InsideOutsideAreaTotal} for an example.
%
%
\begin{figure}
    \centering
    \includegraphics{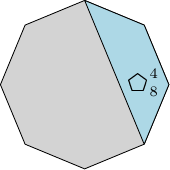}
    \caption{The $8$-gon segment with $4$ corners $\P_8^4$.}
    \label{fig:InsideOutsideAreaTotal}
\end{figure}
Since $\norm{\P_n} = 1$, the size $\norm{\P_n^m}$ of $\P_n^m$ is the fraction of $\P_n$ covered by $\P_n^m$.
In particular, we have $\norm{\P_n^m} \in [0,1]$.
In fact, for most of the results in this paper we take $\alpha = \norm{\P_n^m}$ for an appropriate choice of $m$ (possibly depending on $n$).

\subparagraph*{Our results.}

We investigate whether $\mathcal{G}(S,\alpha)$ has product structure or not.
As soon as $\alpha > 0$, there are indeed constants $w = w(S,\alpha)$ and $k = k(S,\alpha)$, such that all graphs in~$\mathcal{G}(S,\alpha)$ are $K_w$-free and with no $k \times k$-grid in any neighborhood~$N^{1}[v]$.
But, while these obstructions are ruled out,~$\mathcal{G}(S,\alpha)$ has still no product structure if $\alpha > 0$ is too small.

\begin{restatable}{theorem}{resultsnoProductStructureEven}
    \label{thm:results_noProductStructureEven}
    For every $ n \geq 2 $ and $ \alpha < \norm{\P_{2n}^4}$, the class of all intersection graphs of $ \alpha $-free homothetic regular $ 2n $-gons does not admit product structure.
\end{restatable}

In other words, we have $\alpha^*(\P_{2n}) \geq \norm{\P_{2n}^4} > 0$ for every $n \geq 2$.
As an interesting special case, we highlight that $\norm{\P_4^4} = 1$ and hence $\alpha^*(\P_4) = 1$.
That is, for every $\bar{\alpha} > 0$ we can construct collections $\mathcal{C} = \{S_v\}_{v \in V}$ of axis-aligned squares such that each square $S_v \in \mathcal{C}$ has at most an $\bar{\alpha}$-fraction of its area covered by $\mathcal{C}-S_v$, i.e., $\mathcal{C}$ is arbitrarily close to a contact representation, and still the intersection graphs of $\mathcal{C}$ do not admit product structure\footnote{Hence, these graphs do not belong to any graph class with product structure as listed in \textbf{Related work}.}.


In fact, we may encounter the same situation among general (irregular) convex $n$-gons.


\begin{theorem}
    \label{thm:results_noProductStructureNonRegular}
    For every $ n \geq 3 $ there is an $ n $-gon $ S $ such that for all $ \alpha < 1 $ the class of intersection graphs of $ \alpha $-free homothetic copies of $ S $ does not admit product structure.
\end{theorem}


For regular $n$-gons however, if $\alpha$ is large enough, we always have product structure.

\begin{restatable}{theorem}{resultsProductStructure}
    \label{thm:results_productStructure}
    There is an $\alpha < 1 $ such that for
    all $n > 6$, 
    the class of intersection graphs of~$\alpha$-free homothetic regular $n$-gons admits product structure.
\end{restatable}

In other words, we have an $\hat{\alpha} < 1$ such that $\alpha^*(\P_n) \leq \hat{\alpha}$ for every $n > 6$.
We prove \cref{thm:results_productStructure} by showing that the graphs in $\mathcal{G}(\P_n,\hat{\alpha})$ are planar, and for this reason $\mathcal{G}(\P_n,\hat{\alpha})$ admits product structure~\cite{PlanarGraphsQueueNumber}.
To prove the planarity, we consider the \emph{canonical drawing} of the intersection graph $G$, which is derived from an intersection representation with $\alpha$-free homothetic regular $n$-gons by placing each vertex $v$ at the center of its shape $S_v$ and draw each edge $uv$ as a short $1$-bend polyline inside $S_u \cup S_v$.
We define and discuss these drawings more detailed in \cref{sec:embedding}.
In fact, also many beyond-planar graph classes admit product structure (as we shall list below).
And we actually suspect (cf.\ \cref{conj:k-independent-threshold} below) that whether or not $\mathcal{G}(\P_n,\alpha)$ admits product structure is equivalent to whether or not the canonical drawings of the graphs in $\mathcal{G}(\P_n,\alpha)$ belong to a novel type of beyond-planar drawing style.

\begin{definition}
    \label{def:k-independent-crossing}
    For $k \geq 0$, a topological drawing\footnote{Vertices are points and edges are curves connecting their end-vertices. Any two edges have only finitely many points in common; each being either a common endpoint or a proper crossing.} $\Gamma$ of a graph $G$ in $\mathbb{R}^2$ is \emph{$k$-independent crossing} if no edge $e$ of $G$ is crossed in $\Gamma$ by more than $k$ independent edges of $G$.
\end{definition}

Clearly, $0$-independent crossing drawings are precisely planar drawings.
And $1$-independent crossing drawings are precisely fan-crossing drawings~\cite{fan_crossing}.
In general, in a $k$-independent crossing drawing, every edge may for example be crossed by $k$ stars of edges.

\begin{conjecture}
    \label{conj:k-independent-threshold}
    The class of intersection graphs of $\alpha$-free homothetic regular $n$-gons admits product structure if and only if their canonical drawings are $k$-independent crossing for a global constant $k$ (possibly depending on $n$).
\end{conjecture}

Our final contribution is to exactly determine $\alpha$ in terms of $n$ for which the canonical drawings of all graphs in $\mathcal{G}(\P_n,\alpha)$ are $k$-independent crossing for some 
$k$.
For this, we define 
\begin{equation}
    \label{eq:threshold-definition}
    s(n) =
    \begin{cases*}
        \norm{\P_n^{n/2+2}}                 & if $n \equiv 0 \pmod  4$ \\
        \norm{\P_n^{\lceil n/2 \rceil+1}}   & if $n \equiv 1 \pmod  4$ \\
        \norm{\P_n^{n/2+1}} = \frac12       & if $n \equiv 2 \pmod  4$ \\
        \norm{\P_n^{\lceil n/2 \rceil+2}}   & if $n \equiv 3 \pmod  4$.
    \end{cases*}
\end{equation}
The function $ s(n) $ is defined as the tipping point whether or not two regular $\alpha$-free $ n $-gons can meet in a third regular $ n $-gon $ S $ without containing a corner of $ S $.
See \cref{fig:snIdea} for an illustration and
\cref{fig:sn-plot} for a plot of $s(n)$.
In particular note that $s(n) \geq \frac12$ for all $n$ and $\lim_{n\to\infty} s(n) = \frac12$. 
The four cases are due to whether or not the four corners closest to the boundary of $ S $ need to be inside $ S $ in order for the two $ n $-gons to meet, e.g., 
in the case $ n \equiv 0 \mod 4 $ only $ n/2 $ of the corners are inside $ S $ (and therefore, at the tipping point where the two $ n $-gons meet exactly at a corner of $ S $, $ n/2 + 2 $ are outside or at the boundary), while half of the area is covered if $ n/2 + 1 $ corners are inside $ S $ as in the case $ n \equiv 2 \mod 4 $.

\begin{figure}
    \centering
    \includegraphics{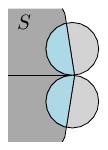}
    \hfill
    \includegraphics{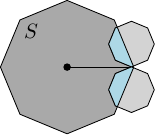}
    \hfill
    \includegraphics{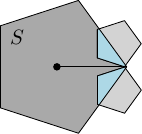}
    \hfill
    \includegraphics{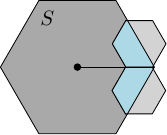}
    \hfill
    \includegraphics{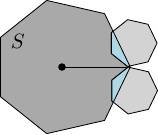}
    \caption{Two regular $ n $-gons meeting in third, $ S $, for large $ n $ and the cases $ n \equiv 0, 1, 2, 3 \mod 4 $.}
    \label{fig:snIdea}
\end{figure}

\begin{figure}
    \centering
    \includegraphics{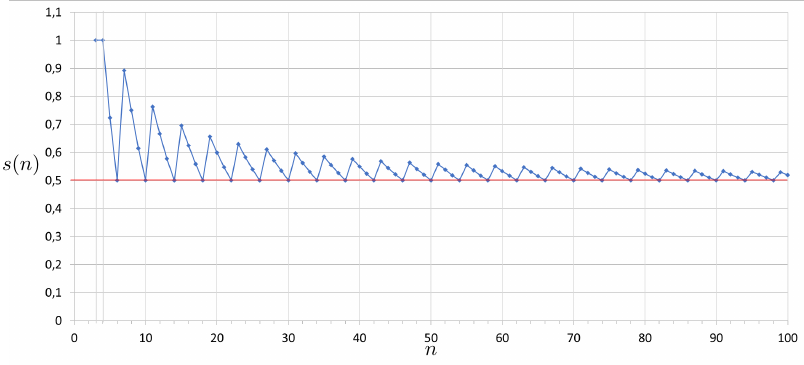}
    \caption{Values of $s(n)$ for $n = 3, \ldots, 100$.}
    \label{fig:sn-plot}
\end{figure}

\begin{restatable}{theorem}{resultskindependentcrossing}
    \label{thm:results_k-independent-crossing}
    Let $n \geq 3$, $\alpha \in [0,1]$, and $s(n)$ be defined as in~\eqref{eq:threshold-definition}.
    Then there exists a constant $k = k(n)$ such that the canonical drawings of all graphs in $\mathcal{G}(\P_n,\alpha)$ are $k$-independent crossing if and only if $\alpha \geq s(n)$.
\end{restatable}

Crucially, \cref{thm:results_k-independent-crossing,conj:k-independent-threshold} together would give $\alpha^*(\P_n) = s(n)$, i.e., that $\alpha \geq s(n)$ is also the exact tipping point for the product structure of the class $\mathcal{G}(\P_n,\alpha)$. An overview of our results for even $n > 6$ is given in \cref{fig:AlphaInterval}.

\begin{figure}
    \centering
    \includegraphics{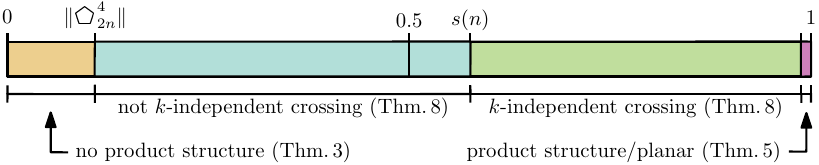}
    \caption{Our results for even $n > 6$. It holds that $\lim_{n\to\infty} s(n) = \frac12$ and $\lim_{n\to\infty} \norm{\P_{2m}^4} = 0$.}       
    \label{fig:AlphaInterval}
\end{figure}

\subparagraph*{Organization of the paper.}

After some related work below, we quickly define in \cref{sec:preliminaries} treewidth and product structure.
In \cref{sec:NoProductStructure} we prove \cref{thm:results_noProductStructureEven,thm:results_noProductStructureNonRegular} by presenting a new graph class called \emph{nested grids} with \emph{no} linear local treewidth (hence no product structure), and then constructing $\alpha$-free intersection representations of nested grids for the claimed sets~$S$.
In \cref{sec:embedding} we define and discuss the canonical drawing of an intersection graph from an $\alpha$-free intersection representation.
The canonical drawings are then used to prove \cref{thm:results_productStructure} in \cref{sec:ProductStructure} and \cref{thm:results_k-independent-crossing} in \cref{sec:ProductStructurePossible}.
We discuss conclusions in \cref{sec:conclusion}.

\subparagraph*{Related work.}

Since the introduction by Dujmovi\'{c} et al.~\cite{PlanarGraphsQueueNumber}, 
a variety of graph classes have been shown to admit product structure, including
planar graphs, graphs with bounded Euler genus $g$, apex-minor-free graphs~\cite{PlanarGraphsQueueNumber, ImprovedPlanarGraphProductStructureTheorem},
$k$-planar graphs, $k$-nearest-neighbor graphs, $(g,k)$-planar graphs, $d$-map graphs, $(g,d)$-map graphs~\cite{ProductStructurekPlanarGraphs},
$h$-framed graphs~\cite{hframedgraphs}, $(g,\delta)$-string graphs~\cite{distel2023powers,ProductStructurekPlanarGraphs},
$k$-th powers of planar graphs with bounded maximum degree~\cite{distel2023powers,ShallowMinorsUSW},
fan-planar graphs, $k$-fan-bundle graphs~\cite{ShallowMinorsUSW}, and
$K_w$-free intersection graphs of unit disks in $\mathbb{R}^2$~\cite{UnitDisks}.
In addition, product structure has been used to investigate different concepts in graphs; sometimes resolving long-standing conjectures.
This includes 
adjacency labeling schemes~\cite{ShorterLabelingSchemesforPlanarGraphs,AdjacencyLabellingforPlanarGraphsAndBeyond,SparseUniversalGraphsForPlanarity},
nonrepetitive colorings~\cite{Dujmovi__2020},
$p$-centered colorings~\cite{Dbski2021},
clustered colorings~\cite{DEMW-23,dujmović_esperet_morin_walczak_wood_2022},
vertex rankings~\cite{VertexRankingPlanarGraphs},
queue layouts~\cite{PlanarGraphsQueueNumber},
reduced bandwidth~\cite{ReducedBandwidth},
comparable box dimension~\cite{DGLTU-22},
neighborhood complexity~\cite{JR-23},
twin-width~\cite{hframedgraphs,KPS-24}, and
odd-coloring numbers~\cite{OddColouringsofGraphProducts}.

On the other hand, there are only very few results for the \emph{non}-existence of product structure.
Besides linear local treewidth (\cref{def:linear-local-treewidth}), a necessary condition for product structure is having bounded layered treewidth~\cite{Bose_2022}.
In fact, bounded layered treewidth implies linear local treewidth~\cite{DUJMOVIC2017111}, making the former the stronger condition.
For proper minor-closed graph classes, both linear local treewidth and bounded layered treewidth are also sufficient conditions for product structure~\cite{PlanarGraphsQueueNumber}.
However, this does not hold for general graph classes as some graph classes with bounded layered treewidth admit no product structure~\cite{Bose_2022}.


\section{Preliminaries}
\label{sec:preliminaries}

\subparagraph*{Treewidth.}

Treewidth is a graph parameter first introduced by Robertson and Seymour~\cite{ROBERTSON1986309} measuring the similarity of a graph to a tree.
Let us define the edge-maximal graphs of treewidth $t$:
For an integer $t \geq 0$, a \emph{$t$-tree} is a graph $H$ that is either $K_{t+1}$, or obtained from a smaller $t$-tree $H'$ by adding one new vertex $v$ with neighborhood $N(v) \subset V(H')$ that induces a clique of size $t$ in $H'$.
Now, the \emph{treewidth} of a graph $G$, denoted as $\tw(G)$, is the minimum $t$ such that $G \subseteq H$ for some $t$-tree $H$.





\subparagraph*{Strong product of graphs.}

The \emph{strong product of graphs} is a combination of the Cartesian product of graphs and the tensor product of graphs.
The vertex-set of the strong product $G \boxtimes H$ of two graphs $G$ and $H$ is defined as $V(G) \times V(H)$.
The edge-set is the union of the edges in the Cartesian and the tensor product of $G$ and $H$.
That is, there is an edge in $G \boxtimes H$ between two vertices $(u,u'),(v,v') \in V(G\boxtimes H)$ if and only if
\[
    u=v,\, u'v' \in E(H) \qquad \text{or} \qquad u'=v',\, uv \in E(G) \qquad \text{or} \qquad uv \in E(G),\, u'v' \in E(H).
\]
\section{Intersection graphs without product structure}
\label{sec:NoProductStructure}

In this section we prove that, for some $\alpha \in [0,1]$ and some $S \subseteq \mathbb{R}^2$, the class $\mathcal{G}(S,\alpha)$ of all intersection graphs of $\alpha$-free homothetic copies of $S$ does not admit product structure.
In particular, we consider for $S$ regular $2n$-gons in \cref{sec:N46NoProductStruture} and irregular $n$-gons in \cref{sec:N3nNoProductStruture}.
Both cases rely on the same general construction, which we describe first in \cref{sec:gridConstruction}.


\subsection{Nested grids}
\label{sec:gridConstruction}

We aim to construct a graph class $\mathcal{G}$ that does \emph{not} have linear local treewidth (cf.~\cref{def:linear-local-treewidth}).
Then, by~\cite[Lemma 6]{DUJMOVIC2017111,PlanarGraphsQueueNumber}, $\mathcal{G}$ admits no product structure.
Note that if $\mathcal{G}$ has linear local treewidth, for each graph $G = (V,E) \in \mathcal{G}$ its treewidth $\tw(G)$ is linearly bounded by its radius $\rad(G) = \min_{v \in V} \min\{ k \mid N^k[v] = V\}$.
We now aim to construct a sequence $G_1,G_2,\ldots$ of graphs with $\rad(G_k) \in O(k)$ but $\tw(G_k) \in \Omega(k^2)$, i.e., where the treewidth is not linear in the radius.
Here, we give a general description of $G_k$, $k \geq 1$, which is then completed in detail depending on the particular polygon~$S$ in \cref{sec:N46NoProductStruture,sec:N3nNoProductStruture}.

\begin{description}
    \item[Step 1: Large grid to ensure small radius.]
    
    We start with a $(k+1) \times (k+1)$-grid with each edge subdivided twice, called the \emph{large grid}.
    For an intersection representation, we use $(k+1)^2$ large homothets of $S$, denoted $c_{i,j}$ with $i,j \in [k+1]$, in a grid pattern representing the grid-vertices, and $4k^2$ smaller homothets of $S$, called the \emph{subdivision shapes}, for the subdivision-vertices.
    The exact placement is chosen such that the subdivision shapes meet at their corners so that the resulting intersection graph is the desired subdivided grid as shown in \cref{fig:SubdividedGrid}.
    Note that we do not require that the shapes $c_{i,j}$ have only a point contact with the subdivision shapes.
    We refer to this graph as $G_{k,1}$ and to the areas that are bounded by exactly twelve shapes as \emph{cells}.

    \begin{figure}[ht]
        \includegraphics{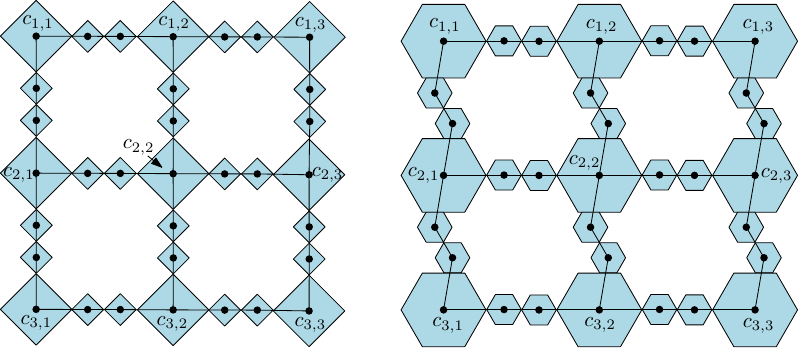}
        \caption{Examples of the large grid with homothetic squares or hexagons, as described in Step~1.}
        \label{fig:SubdividedSquareGrid}
        \label{fig:SubdividedHexagonGrid}
        \label{fig:SubdividedGrid}
        \end{figure}

    \item[Step 2: Small grids inside cells.] 

    Next, we insert a $k\times k$-grid into each cell such that these \emph{small grids} do not touch the large grid from Step~1.
    We denote the union of $G_{k,1}$ and all small grids by $G_{k,2}$.
    In the next step, we connect the small grids so that together they form a grid of size $k^2 \times k^2$ plus some additional edges and subdivisions, yielding quadratic treewidth.
    To do so, we also specify the placement of the small grid inside the cell more precisely in the next step.
    The large grid ensures that the radius of the resulting graphs is linear in $k$.

    \item[Step 3: Connecting the small grids to ensure large treewidth.]

    We connect any two small grids in neighboring cells of $G_{k,1}$ by adding $k$ pairwise disjoint paths, called \emph{connecting paths}, as illustrated in \Cref{fig:SubdividedSquareGridWithPaths}. 
    Each set of connecting paths crosses the subdivided edge of $G_{k,1}$ at the contact point of the two corresponding subdivision shapes.
    To realize these crossings, we must ensure that a contact point of two same-sized homothets of $S$ can be crossed by $k$ independent edges, while keeping all shapes $\alpha$-free.
    Furthermore, we connect the endpoints of these $k$ independent edges to their corresponding small grid with pairwise disjoint paths, while keeping the radius small, that is in $O(k)$.
    
    \begin{figure}
        \centering
        \includegraphics{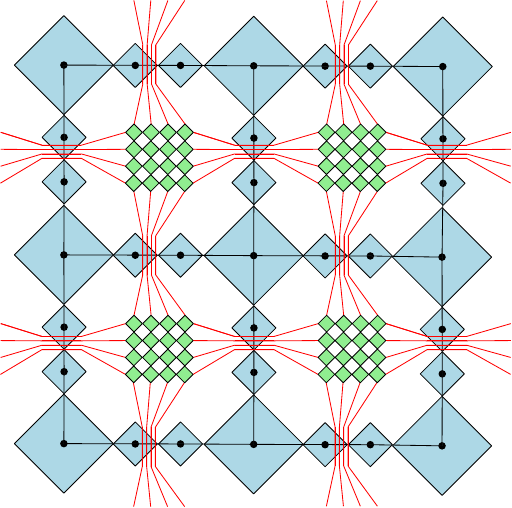}
        \caption{Parts of a large grid (blue) with small grids in its cells (green) connected by paths (red)}        
        \label{fig:SubdividedSquareGridWithPaths}
    \end{figure}
    
    Leaving the shape-specific details of the crossings to \cref{sec:N3nNoProductStruture,sec:N46NoProductStruture}, we now show how to achieve the linear radius.
    As $G_{k,1}$ has radius $\mathcal{O}(k)$, it suffices to show that every vertex we add in a cell has distance $\mathcal{O}(k)$ to some vertex of $G_{k,1}$.
    We achieve this by placing the connecting paths inside a cell within a narrow corridor very close to the border of the cell (\cref{fig:RectangleSmallGrid}).
    Such a corridor and connecting paths along the border of a cell can always be constructed using $1$-free homothets of $S$ of very small size.
    We start with a set $ \mathcal{P}_t $ of $ k $ paths connecting the top of the cell with the small grid, which is placed near the top of the cell for this purpose.
    All further connecting paths in the same cell can be placed iteratively by going along the new boundary.
    Note that the small grid is not necessarily placed in the center of the cell as the exact geometry inside the cell is shape-specific and the center might not be reachable while keeping the radius small.
    As we create $4k$ paths per cell, every vertex on the paths has distance $O(k)$ from the grid $G_{k,1}$. 
    In addition, the small grids are placed such that they touch the end of the constructed paths, yielding a distance of $O(k)$ for every vertex in each cell.

    \begin{figure}[htp] 
        \centering
        \includegraphics[width=\textwidth]{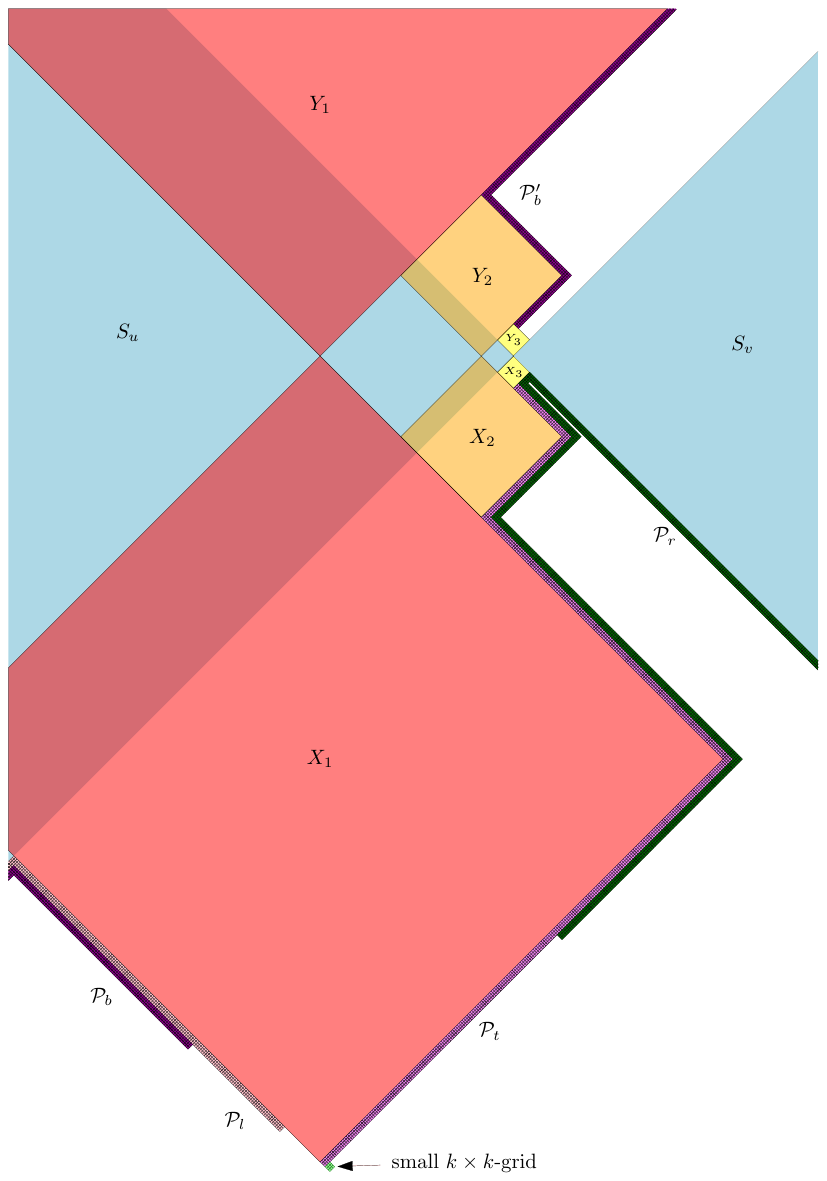}
        \caption{
            Example of how the paths connecting the grids can be constructed.
            $S_u$ and $S_v$ are subdivision shapes and
            $X_1, Y_1, X_2, Y_2 $, and $ X_3, Y_3$ realize the independent edges crossing the contact point of $ S_u $ and $ S_v $ required in Step 3.
            $ \mathcal{P}_t $ connects the top of the cell, i.e.\ $ X_1, X_2, X_3 $, with the small grid.
            $ \mathcal{P}_l, \mathcal{P}_b, \mathcal{P}_r $ connect the left, bottom, and right side of the cell to the small grid and are shortened for improved readability.
            Note that the latter three sets of paths are not symmetric but walk along the boundary of the cell in order to reach the small grid while keeping the radius small.
        }
        \label{fig:RectangleSmallGrid}
    \end{figure}
\end{description}

\subsection{Regular \texorpdfstring{$\bm{2n}$}{2n}-gons}
\label{sec:N46NoProductStruture}





In this section we prove \cref{thm:results_noProductStructureEven} by giving for every $n \geq 2$ an explicit $\alpha < 1$ such that intersection graphs of $\alpha$-free homothetic copies of $\P_{2n}$ admit no product structure.
Recall that $\norm{\P_{2n}^m} \in [0,1]$ denotes the portion of the area of $\P_{2n}$ within a segment with $m$ corners.

\begin{theorem} \label{thm:Regular2NGons}
    For every $n \geq 2$ and $\alpha < \norm{\P_{2n}^4}$, the class of all intersection graphs of $\alpha$-free homothetic regular $2n$-gons does not admit product structure.
\end{theorem}

Hence, as $\norm{\P_4^4}=1$ and $\norm{\P_6^4}=\frac12$, regular $\alpha$-free squares do not admit product structure for any $\alpha < 1$, and regular $\alpha$-free hexagons do not admit product structure for any $\alpha < \frac12$.

\begin{proof}
    We prove the theorem by using the construction described in \Cref{sec:gridConstruction}.
    Thus, we need to show that all three steps of the construction are feasible using $\alpha$-free regular $n$-gons.
    Constructing the grid $G_{k,1}$ described in Step~1 using regular $n$-gons is clearly possible for all $n$, e.g., see \Cref{fig:SubdividedGrid} for $n=4$ and $n=6$.
    For Step~2 place a $k \times k$-grid inside each cell of $G_{k,1}$ yielding $G_{k,2}$.
    The main challenge is to show that Step~3 of the construction is feasible. 
    
    In Step~3 we connect the small grids to together contain a $k^2\times k^2$-grid subdivision. 
    Let $S_u, S_v$ be two adjacent subdivision shapes in the grid $G_{k,1}$ and recall that they meet at a corner.
    We aim to construct $k$ pairwise disjoint paths crossing the $S_u$-$S_v$-contact with contact point $q$, see \cref{fig:CrossingHexagons}~left.
    We thereby ensure that the inserted shapes do not intersect any shapes other than $ S_u $ and $ S_v $.
    \begin{figure}%
        \centering%
        \includegraphics[width=\textwidth]{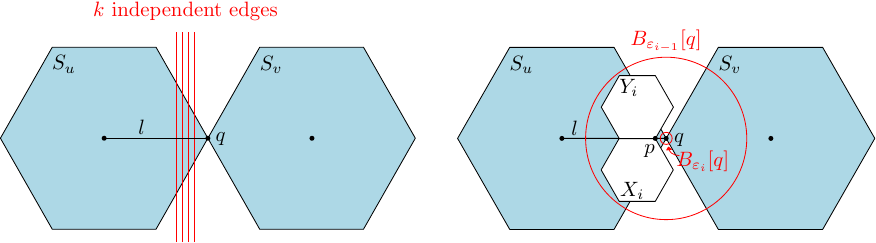}%
        \caption{%
            Left: We aim to cross two hexagons $U,V$ with $k$ independent edges.
            Right: Two hexagons $X_i,Y_i$ inside the ball $B_{\epsilon_{i-1}}[q]$ with $c^{X_i}_{n}=p=c^{Y_i}_2$ that cross the line segment $l$ from the center of $S_u$ to $q$. In the next iteration, the hexagons $ X_{i+1}, Y_{i+1} $ are placed inside $B_{\epsilon_{i}}[q]$.
        }%
        \label{fig:CrossingHexagons}%
        \label{fig:CrossingHexagons2}%
    \end{figure}
    We iteratively place $n$-gons $X_1, Y_1, \dots , X_k, Y_k$ such that after iteration $i$ with $i \in [k]$, the $n$-gons $X_1Y_1, \dots , X_iY_i$ form $i$ independent edges crossing the line segment $l$ from the center of $S_u$ to $q$. 
    Additionally, after iteration $i$ there is an $ \epsilon_i $-ball $B_{\varepsilon_i}[q]$  around $q$ for some $\epsilon_i > 0$ such that $B_{\varepsilon_i}[q]$ does not intersect any $X_j$ or $Y_j$ for $j \leq i$. 
    For ease of presentation, let $ B_{\epsilon_0}[q] $ be a ball around $ q $ that only intersects $ S_u $ and~$ S_v $. 

    In iteration $i$ we consider the ball $B_{\varepsilon_{i-1}}[q]$ that does not intersect any $n$-gons $X_j, Y_j$ placed before. Let the corners of an $n$-gon $S_z$ be $c_1^z, \dots, c_n^z$ in clockwise order. Without loss of generality, let $q$ be the corner $c^u_1$ of $S_u$. 
    Let $c^{X_i}_1$ be the corresponding corner of the $n$-gon $X_i$ and $c^{Y_i}_1$ be the corresponding corner of $Y_i$. 
    Further let $p$ be a point close to $ q $ on $l$ with $\epsilon' $ distance from $q$, where $ 0 < \epsilon' < \epsilon_{i-1}$. 
    We place $X_i, Y_i$ inside $B_{\varepsilon_{i-1}}[q]$ such that they share a side and meet with a corner at $ p $, i.e., such that $c^{X_i}_{n}=p=c^{Y_i}_2$ as shown in \cref{fig:CrossingHexagons2}~right. 

    
    Note that two corners of $X_i$ and $Y_i$, respectively, are placed outside of $S_u$ and the two shapes do not intersect $ S_v $.
    In addition, two corners are placed $ \epsilon'$-close to the border of $S_u$. 
    Thus, for $\epsilon'$ small enough $X_i$ and $Y_i$ have arbitrarily less than $\norm{\P_{2n}^{4}}$ area disjoint from $S_u$.
    As $ X_i $ and  $ Y_i $ do not intersect any other shapes than $ S_u $, we can choose $ \epsilon' $ sufficiently small such that $ X_i $ and $ Y_i $ are $ \alpha $-free for every $ \alpha < \norm{\P_{2n}^{4}}$.
    Further observe that as $ \epsilon' > 0 $, after placing $X_i$ and $Y_i$ there still exists a ball $B_{\varepsilon_i}[q]$ with positive radius $ \epsilon_i $ around $q$ that does not intersect any $X_j$ or $Y_j$ for $j \leq i$. 
    Thus, the invariants hold and we can continue placing $n$-gons. 
    After $k$ iterations we have $k$ independent edges crossing $S_u$ and $S_v$, as required.

    Finally, the edges $X_1Y_1, \dots , X_iY_i$ can be connected to the small grids as explained in Step 3 of the construction while keeping the radius in $\mathcal{O}(k)$. 
    Thus, we have constructed a graph class that does not have linear local treewidth, which rules out product structure.
\end{proof}

\subsection{Triangles and irregular \texorpdfstring{$\bm{n}$}{n}-gons}
\label{sec:N3nNoProductStruture}


This section is devoted to \cref{thm:results_noProductStructureNonRegular}, which states that for every $n \geq 3 $ there is a (possibly non-regular) $n$-gon $ S $ such that $\alpha$-free intersection graphs of shapes homothetic to $ S $ do not admit product structure for any $\alpha < 1$.
As all triangles are affinely equivalent, we conclude:

\begin{corollary}
    The graph class of $\alpha$-free intersection graphs of homothetic triangles does not admit product structure for any $\alpha < 1$.
\end{corollary}

The following \lcnamecref{lem:nonRegularNGons} specifies the shapes we use and immediately implies \cref{thm:results_noProductStructureNonRegular}.
We refer to 
\cref{fig:SubdividedTriangleGrid} 
for examples of shapes that satisfy the required properties.
The main difference to \cref{sec:N46NoProductStruture} is how we implement crossings, which is shown in \cref{fig:CrossingTriangles_shortVersion}.

\begin{figure}
		\centering
        \includegraphics[page=2]{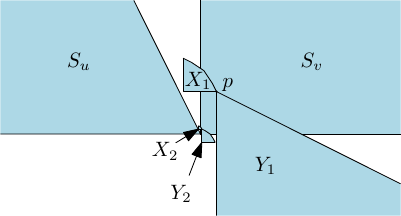}
        \caption{%
            Two shapes $S_u, S_v$ that we cross with two independent edges $ x_1 y_1 $, $ x_2 y_2 $.
            Note that to the bottom-left of $ X_2 $ and $ Y_2 $, there is space for more crossing edges.
        }
        \label{fig:CrossingTriangles_shortVersion}
    \end{figure}

\begin{restatable}{lemma}{nonRegularNGons}
    \label{thm:nonRegularNGons}\label{lem:nonRegularNGons}
    Let $S$ be a convex shape with two orthogonal adjacent sides $l(S),b(S)$ such that $S$ is contained in the rectangle spanned by $l(S)$ and $b(S)$ and no sides of $S$ are parallel to $l(S)$ or $b(S)$. Then, for no $\alpha  < 1$ does the class of all $\alpha$-free intersection graphs of shapes homothetic to $S$ admit product structure.
\end{restatable}

\begin{proof}
    We again use the construction described in \cref{sec:gridConstruction}. 
    So we aim to show that all three steps of the construction are possible using the shape $S$. 
    In the following we refer to the corner of $S$ contained in both $b(S)$ and $l(S)$ as $v_{bl}(S)$, and to the corner of $S$ contained in $b(S)$, respectively $ l(S)$, but not in the other as $v_b(S)$, respectively, $ v_l(S)$.
    For better readability, we draw all figures so that side $ l(S) $ is to the left, $ b(S) $ is the bottom side, and therefore $ v_l(S) $ is the topmost corner, $ v_b(S) $ is the rightmost corner, and $ v_{lb}(S) $ is bottom-left.

    As before, Steps~1 and~2 are straight-forward to implement.
    We refer to \cref{fig:SubdividedTriangleGrid} for an illustration of a cell of the resulting large grid $ G_{k,1} $ with a small grid inside.
    Note that adjacent shapes only have point contacts, as otherwise $S$ would have a side parallel to $l(S)$ or $b(S)$, which is prohibited.

    \begin{figure}
        \centering
        \includegraphics[scale=1]{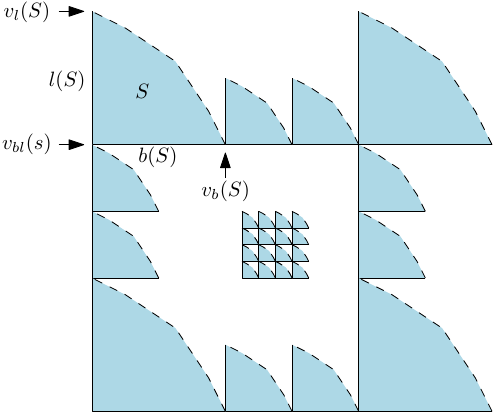}
        \caption{A cell of the grid constructed in Step~1 of \cref{lem:nonRegularNGons} using a shape $S$. The dashed sides may look different. The cell contains a $4\times4$-grid constructed in Step~2.}
        \label{fig:SubdividedTriangleGrid}
    \end{figure}

    To show that Step~3 is feasible we connect the small grids contained in each cell to form a $k^2\times k^2$-grid subdivision. 
    Let $S_u, S_v$ be two adjacent subdivision shapes in the grid $G_{k,1}$. 
    We aim to construct $k$ pairwise disjoint paths crossing the $S_u$-$S_v$-contact from one cell to another without intersecting any shapes other than $ S_u $ and $ S_v $.
    Assume without loss of generality that $S_u$ and $S_v$ meet at the corners $v_b(S_u)=v_{bl}(S_v)$ as shown in \Cref{fig:CrossingTriangles}. 
    The other cases are symmetrical.

    \begin{figure}
		\centering
        \includegraphics[page=2]{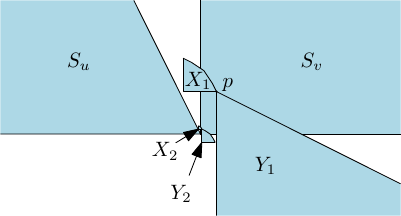}
        \caption{%
            Two shapes $S_u,S_v$ that we cross with $k$ independent edges.
            The corners $v_l(Y_1)$ and $ v_b(X_1) $ meet at $ p $ inside $ S_v $.
            We continue to the bottom-left with the smaller shapes $ X_2 $ and $ Y_2 $ forming the second edge.
        }
        \label{fig:CrossingTriangles}
    \end{figure}
    
    
    Next we introduce two shapes $ X_1 $ and $ Y_1 $ crossing the $ S_u $-$ S_v $-contact.
    Both shapes are placed such that they are disjoint from $ S_u $ and only one of their corners is in $ S_v, $ namely the rightmost corner $ v_b(S_u) $ for $ S_u $ and the topmost corner $ v_l(S_v) $ for $ S_v, $ and these two corners meet at a point $ p $ (\cref{fig:CrossingTriangles}).
    Note that $ p $ can be placed close enough to the boundary of $ S_v, $ without actually hitting the boundary, so that all shapes are $ \alpha $-free.
    
    We repeat the construction and add a total of $2k$ shapes $ X_1, Y_1, \dots, X_k, Y_k $.
    With each step of the iteration the shapes $X_i$ and $Y_i$, $ i \in [k] $, become smaller so that they do not intersect.
    Observe that we always find space for the next iteration as the contact point of $ X_i $ and $ Y_i $ is strictly inside $ S_v. $
    Repeating the construction $k$ times we get the independent $k$ edges crossing $S_u$ and $S_v$ required by Step~3, while keeping each shape $\alpha$-free.

    Finally, these edges can be connected to the small grids in the cells as explained in \cref{sec:gridConstruction} while keeping the radius in $\mathcal{O}(k)$. 
    Thus, we obtain a graph class with no linear local treewidth, and hence no product structure.
\end{proof}

\section{Canonical drawings}
\label{sec:embedding}

In this section, we describe how we derive a drawing of the corresponding intersection graph $G = (V,E)$ from a collection $\mathcal{C}$ of $\alpha$-free homothetic copies of $\P_n$ (for some $\alpha > 0$).
That is, we identify a point in $\mathbb{R}^2$ for each vertex $v \in V$ inside its corresponding set $S_v \in \mathcal{C}$, and route each edge $uv \in E$ as a polyline in $\mathbb{R}^2$ inside an $\varepsilon$-blowup of $S_u \cup S_v$.
Both steps are quite natural, but some care is needed in the details.
While there is nothing surprising here, in the upcoming \cref{sec:ProductStructure,sec:ProductStructurePossible} we prove that for specific choices of $n$ and $\alpha$, these drawings have interesting properties, such as being planar or $k$-independent crossing (cf.~\cref{def:k-independent-crossing}).

\medskip

Let $n \geq 3$ be fixed, and $\mathcal{C} = \{S_v\}_{v \in V}$ be a collection of $\alpha$-free homothetic copies of $\P_n$ for some $\alpha > 0$, and $G = (V,E)$ be its intersection graph.
Choose $\varepsilon > 0$ small enough (to be discussed later).
For each shape $S_v \in \mathcal{C}$ let $c_v$ denote its center.
We draw each vertex $v \in V$ as a point inside the $\epsilon$-ball $B_\varepsilon(c_v)$ around $c_v$, such that all vertices lie in general position.

\begin{figure}[ht]
    \centering
    \includegraphics[page=1]{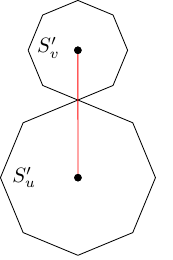}
    \hfill
    \includegraphics[page=2]{embedding.pdf}
    \hfill
    \includegraphics[page=3]{embedding.pdf}
    \hfill
    \includegraphics[page=4]{embedding.pdf}
    \caption{Embeddings with the edges (thick red) inside the shapes}
    \label{fig:embedding}
\end{figure}

Now, for every edge $uv \in E$ individually, we do the following.
First, we scale down $S_u$ at its center $c_u$ to $S'_u$, and $S_v$ at its center $c_v$ to $S'_v$, such that $B_\varepsilon(c_u) \subset S'_u$, $B_\varepsilon(c_v) \subset S'_v$, and $S'_u$ and $S'_v$ touch but share no positive area.
If the line segment $\overline{c_u c_v}$ intersects $S'_u \cap S'_v$, let $p_{uv}$ denote the point of intersection.
This is for example always the case when $n$ is even (see \cref{fig:embedding}).
Otherwise, let $p_{uv}$ be the single point in $S'_u \cap S'_v$ (see \cref{fig:embedding}-right).

We now draw the edge $uv$ as a $1$-bend polyline connecting $u$ to $p_{uv}$ and $p_{uv}$ to $v$. 
Observe that the edge $uv$, including its bend $p_{uv}$, is drawn inside $S'_u \cup S'_v$ and hence inside $S_u \cup S_v$.
In case the bend points of several edges happen to coincide, we slightly move the bend points within their $\epsilon$-balls such that no two such edges with a common endpoint cross.
Similarly, we slightly move the bend points such that they are in general position together with the vertices.
Hence, the drawing is simple\footnote{An embedding is called \emph{simple} if vertices and edges do not share points except for incident edges meeting at their common endpoint and non-adjacent edges may cross once but only two in a point (i.e., no touchings, no self-crossings, no crossings of adjacent edges, no three edges crossing in the same point).} except that edges may cross twice.

\medskip

Let us list some crucial properties of the resulting drawing.

\begin{observation}
    \label{obs:drawing}
    Given $\alpha$-free homothetic copies of $\P_n$, the canonical drawing $\Gamma$ of their intersection graph $G = (V,E)$ satisfies the following properties:
    \begin{itemize}
        \item Every vertex $v \in V$ is drawn $\varepsilon$-close to the center $c_v$ of its shape $S_v$.
        \item For every edge $uv \in E$ there are scaled-down interiorly disjoint $S'_u,S'_v$ with $B_\varepsilon(c_u) \subset S'_u \subseteq S_u$ and $B_\varepsilon(c_v) \subset S'_v \subseteq S_v$, with the edge drawn as a polyline with its only bend $\varepsilon$-close to $p_{uv} \in S'_u \cap S'_v$.
        \item The set of all bend points and all vertices is in general position.
    \end{itemize}
\end{observation}

Note that we choose the $\epsilon$-offsets sufficiently small so that if there is a crossing in our drawing, then the two edges also intersect in the possibly non-simple embedding obtained by choosing $\epsilon = 0$.
Hence, from now on we may assume the vertices and bends to be placed exactly at the centers, respectively contact points, for checking whether two edges cross.

\section{Planar drawings}
\label{sec:ProductStructure}

Complementing the results in \cref{sec:NoProductStructure}, we show here that for some $\alpha < 1$ and all $n > 6$, the $\alpha$-free intersection graphs of homothetic regular $n$-gons admit product structure.
%
%
But let us quickly discuss the case $\alpha = 1$ first.
Here we have contact representations, i.e., the $n$-gons are interiorly disjoint and induce an edge if they touch.
For every $n \neq 4$, these contact graphs are planar, and hence admit product structure~\cite{PlanarGraphsQueueNumber}.
For $n = 4$, we have contact graphs of axis-aligned squares, which are $1$-planar, and hence also admit product structure~\cite{ProductStructurekPlanarGraphs}.

\begin{observation}\label{obs:alpha1}
    For every $n \geq 3$, the class of $1$-free intersection graphs of homothetic regular $n$-gons admits product structure.
\end{observation}

Turning back to the case $\alpha < 1$, i.e., the statement of \cref{thm:results_productStructure}, we shall use the canonical drawings defined in \cref{sec:embedding}.
To prove \cref{thm:results_productStructure}, we show that for appropriate $\alpha < 1$ and all $ n > 6 $ these canoncial drawings are crossing-free and thus the corresponding class of intersection graphs admits product structure by~\cite{PlanarGraphsQueueNumber}.

\begin{lemma} \label{lem:Crossing3Touching}
    Let $n \geq 3$, $\alpha < 1$, and $G$ be an intersection graph of $\alpha$-free homothetic regular $n$-gons with canonical drawing $\Gamma$.
    If two edges $uv,xy \in E$ cross in $\Gamma$, then there is a point $p \in \mathbb{R}^2$ that is contained in at least three of $S_u,S_v,S_x,S_y$.
\end{lemma}

\begin{proof}
    Consider the scaled-down $n$-gons $S'_u$ and $S'_v$ used to draw the edge $uv$, as well as $S'_x$ and $S'_y$ used to draw the edge $xy$.
    In particular, consider $p_{uv} \in S'_u \cap S'_v$ and $p_{xy} \in S'_x \cap S'_y$.
    Further, we may assume that the crossing of $uv$ and $xy$ involves the segments $\overline{u p_{uv}}$ and $\overline{x p_{xy}}$.
    Now observe that if $p_{xy}\in S_u$, then $p_{xy} \in S_u \cap S_x \cap S_y$ and we are done.
    Similarly, we are done if $p_{uv} \in S_x$.
    In the remainder, we aim to show that one of the two cases applies.

    \begin{figure}
	\centering
	\includegraphics{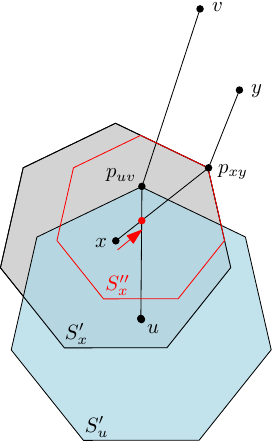}
        \caption{Situation in the proof of \Cref{lem:Crossing3Touching} with $p_{uv} \in S'_x \subseteq S_x$.}
        \label{fig:scaled}
    \end{figure}

    For this, let $S''_x$ be obtained from $S'_x$ by scaling it down at the point $p_{xy}$ until the center of $S''_x$ lies on the segment $\overline{u p_{uv}}$ as shown in \cref{fig:scaled}.
    Since $p_{xy} \in S'_x$, we have $p_{xy} \in S''_x$ and $S''_x \subseteq S'_x \subseteq S_x$. 
    As the center of $S''_x$ lies on the segment $\overline{u p_{uv}}$, we obtain that $S''_x$ either contains $p_{uv}$ or is completely contained in $S'_u \subseteq S_u$.
    In the first case, we have $p_{uv} \in S''_x \subseteq S_x$, while in the second we have $p_{xy} \in S''_x \subseteq S_u$, as desired.
\end{proof}

By \cref{lem:Crossing3Touching}, crossings in $\Gamma$ are only possible if three homothetic copies of $\P_n$ have a common point.
However, this in turn (as long as $n > 6$) forces that one of the three shapes has some significant portion of its area covered by the other two.

\begin{restatable}{lemma}{threeTouchingOverlap}
    \label{lem:3touching-overlap}
    There is an $ \alpha < 1 $ such that for every $n > 6$ and every three homothetic copies $S_u,S_v,S_w$ of $\P_n$ that have a common point $p \in S_u \cap S_v \cap S_w$, one of the three copies is not $ \alpha $-free.
\end{restatable}

\begin{proof}
    First consider the case that $S_u,S_v,S_w$ have the same size and meet at a point $p$ on their boundaries.
    Each interior angle of $S_u,S_v,S_w$ at $p$ is at least $\frac{n-2}{n}\pi > \frac23 \pi$, since $n > 6$.
    Hence, these three angles sum up to more than $2\pi$ and thus the minimum overlap is positive and converges toward the overlap of three circles evenly distributed around $p$ as $n$ grows, see \cref{fig:overlap_boundary}.
    For $n$ divisible by $6$, a sixth of each boundary is covered by each of the other two shapes.
    That is, the portion of $S_u$ that is covered by $S_v \cup S_w$ is $4 \cdot \norm{\P_n^{n/6+1}}$, which tends to $4 (\frac{1}{6} - \frac{\sqrt{3}}{4 \pi}) \approx 0.115$ as $n \to \infty$ (cf.\ \cref{lem:Amn} for a formula for the area).

    \begin{figure}
        \centering
        \includegraphics[scale=0.9]{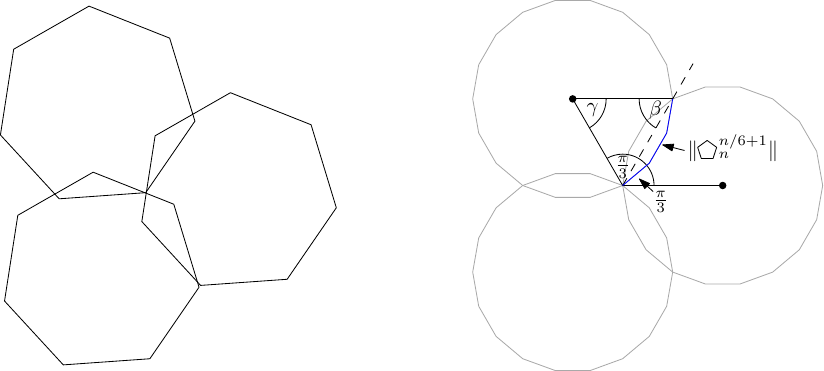}
        \caption{Three 7-gons, respectively three 18-gons, meeting in a point on their boundaries, where each shape has positive overlap. The overlap is minimized by three shapes of the same size with angles $2\pi/3$ between each two centers. By symmetry, the angle $\beta$ is $\pi/3$, and hence also $\gamma$, which is why the blue polygon has $n/6+1$ corners and area $\norm{\P_n^{n/6+1}}$.}
        \label{fig:overlap_boundary}
    \end{figure}
    
    Second, if the sizes of $S_u,S_v,S_w$ differ but $p$ is still on each boundary, then (compared to the previous case) the overlap of the smallest shape increases\footnote{For improved readability, we use increase and decrease in this proof even if sometimes the overlap may also not change.} or the angle between the centers of the other two shapes with respect to $ p $ decreases, which increases their intersection.
    
    In the remainder, we show that it is indeed optimal to have the common point $p$ on all three boundaries.
    For this, we first show that, there is always an alternative point $p' \in S_u\cap S_v\cap S_w$ that is on the boundary of two of $S_u,S_v,S_w$.
    Indeed, we find such a point $p'$ one the boundary of the (possibly degenerate) polygon $S_u\cap S_v\cap S_w$.
    (If $S_u\cap S_v\cap S_w$ contains no such point, then one shape is contained in another, and we can choose $\alpha = 1$.)
    
    Having $p'$ on the boundary of $S_u$ and $S_v$ but in the interior of $S_w$, we next move $S_w$ until $p'$ is on its boundary while decreasing the overlap.
    For this, consider the free area of $S_w$, which consists of one or two components (again, otherwise we can choose $\alpha \neq 1$).
    If there is only one component, then we move $S_w$ so that this component increases, thereby decreasing the overlap.
    On the other hand, if we have two components, then their size depends quadratically on their diameter.
    That is, the free area is minimized by choosing the two components to have the same size and maximized by choosing one of the two components to have zero size.
    Choosing an optimal position for $S_w$ thus ends up with $p'$ on its boundary and hence in a previous case.
    In both cases, we move $S_w$ under the constraint that the ratio of the area $S_w$ shares with each of the other two shapes stays the same.
    Hence, minimizing the overlap of $S_w$ also decreases the overlap of the other two.
    Together, we obtain that the overlap is indeed minimized by moving $S_w$ such that the three shapes meet in a common point on their boundaries, which concludes the proof.
\end{proof}

\Cref{lem:3touching-overlap} together with \cref{lem:Crossing3Touching} implies \cref{thm:results_productStructure}.

\resultsProductStructure*

We finish this section by calculating the area $\norm{\P_n^m}$ for $ m \leq n/2 $, which is not strictly necessary to verify our statements but sometimes convenient to know, e.g., in the proof of \cref{lem:3touching-overlap}. 
Recall that $\P_n$ denote a regular $n$-gon with area $\norm{\P_n} = 1$.
Moreover, for $m \leq n$ we denote by $\P_n^m$ a subset of $\P_n$ that is the convex hull of $m$ consecutive corners of $\P_n$.
Note that for $ m > n/2 $, the area is obtained by $ \norm{\P_n^m} = 1 - \norm{\P_n^{n - m + 2}} $.

\begin{lemma}
    \label{lem:Amn}
    For a regular $n$-gon with area $1$, the area $\norm{\P_n^m}$ of an $n$-gon segment with $m \leq n/2 $ corners is
    \begin{equation*}
        \norm{\P_n^m} = \frac{(m-1) \sin \theta - \sin ((m-1)\theta)}{n \sin(\theta)},
    \end{equation*}
    where $\theta = \frac{2\pi}{n}$ is the angle at the center of $\P_n$ between two consecutive corners of $\P_n$.
\end{lemma}

\begin{proof}
    For ease of presentation, we calculate the area $\norm{S_n^m}$ of a segment $S_n^m$ with $m$ corners of a regular $n$-gon $S$ with cirumradius $1$, and then divide by the area of $S$ at the end.
    The area of $S_n^m$ is obtained from the area of the disk segment with angle $\eta = (m-1)\theta$ by subtracting $m-1$ times the area of a disk segment with angle $\theta$ (\cref{fig:app_InsideOutsideAreaTotal}).
    \begin{figure}
        \centering
        \includegraphics[page=2]{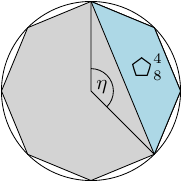}
        \caption{
            Area $\norm{\P_n^m}$ of an $n$-gon segment (blue) with $m=4$ corners.
            The area of $\P^m_n $ is obtained from the area of the disk segment with angle $ \eta = (m-1) \theta $ by subtracting $ m - 1 $ times the shown disk segments with angle $\theta = 2\pi/n$.
        }
        \label{fig:app_InsideOutsideAreaTotal}
    \end{figure}
    Here, a \emph{disk segment} is the shape bounded by a chord of a circle and the shorter of the two arcs, where its angle is given by the angle between the two endpoints of the chord with respect to the center of the circle.
    As the area of a disk segment with radius 1 and angle $ \beta $ is $ \area_\beta = (\beta - \sin \beta)/2 $, we obtain   
    \begin{align*}
        2 \cdot \norm{S_n^m} &= 2 \area_\eta - 2 (m-1) \area_\theta \\
        &= \eta - \sin \eta - (m-1) (\theta - \sin \theta) \\
        &= (m-1)\theta - \sin ((m-1)\theta) + (m-1) \theta - (m-1) \sin \theta \\
        &= (m-1) \sin \theta - \sin ((m-1)\theta).
    \end{align*}
    Dividing this by two times the area $ n \sin(\theta)/2 $ of a regular $n$-gon with circumradius 1 yields the desired formula.
\end{proof}

\section{\texorpdfstring{$\bm{k}$}{k}-independent crossing drawings}
\label{sec:ProductStructurePossible}

In this section we again consider intersection graphs of $\alpha$-free homothetic regular $n$-gons, and specifically, whether their canonical drawings are $k$-independent crossing (cf.~\cref{def:k-independent-crossing}) for a global constant $k$ (that might depend on $n$).
For fixed $n \geq 3$, we let $\alpha \in [0,1]$ vary.
For $\alpha = 1$, one can show that the canonical drawings are planar (or $1$-planar for $n=4$) and in particular $1$-independent crossing.
For smaller $\alpha$, we have a richer graph class, which is less likely to have $k$-independent crossing canonical drawings for any constant $k$.
In fact, we shall prove that $s(n)$ as defined in~\eqref{eq:threshold-definition} is the precise tipping point for $\alpha$ until which the canonical drawings for $\mathcal{G}(\P_n,\alpha)$ are $k$-independent crossing.
That is, we prove \cref{thm:results_k-independent-crossing}.



\subsection{Not \texorpdfstring{$\bm{k}$}{k}-independent crossing for \texorpdfstring{$\bm{\alpha < s(n)}$}{alpha < s(n)}}
\label{sec:BoundarySn}

We show that for $ \alpha < s(n) $, edges can be crossed by arbitrarily many independent edges.


\begin{proposition}
    \label{prop:alpha-small-no-independent}
    Let $n \geq 3$, $\alpha < s(n)$, and $k \geq 1$.
    Then there is a collection $\mathcal{C}_k = \{S_v\}_{v \in V}$ of $\alpha$-free homothetic regular $n$-gons with intersection graph $G_k = (V,E)$ such that one particular edge $uv \in E$ is crossed in the canonical drawing $\Gamma$ of $G_k$ by $k$ independent edges.
\end{proposition}

\begin{proof}
    For the case $n \equiv 0 \pmod 4$, we rotate the regular $n$-gon $\P_n$ such that it has four corners at its extreme $x$- and $y$-coordinates; a top, a bottom, a left, and a right corner.
    
    We start by placing a homothetic regular $n$-gon $S_u \subseteq \mathbb{R}^2$ with center $c$ and right corner $q$.
    We iteratively place $n$-gons $X_1,Y_1,\ldots,X_k,Y_k$ such that after step $i$, $1 \leq i \leq k$ the $n$-gons $X_1,Y_1,\ldots,X_i, Y_i$ are placed and the $i$ corresponding independent edges $x_1y_1,\ldots,x_iy_i$ all cross the line segment $l$ from the center $c$ of $S_u$ to $q$.
    Additionally, there are $\varepsilon_0 > \varepsilon_1 > \cdots > \varepsilon_k$ such that the $\varepsilon_i$-ball $B_{\varepsilon_i}[q]$ around $q$ is disjoint from $X_j$ and $Y_j$ whenever $j \leq i$.
    Clearly, all invariants hold before step~$1$ with $\varepsilon_0 > 0$ being any value small enough such that at least an $\alpha$-fraction of $S_u$ is not covered by $B_{\varepsilon_0}[q]$.

    In step~$i$, we consider the ball $B_{\varepsilon_{i-1}}[q]$ around $q$ that is disjoint from $X_1,Y_1,\ldots,X_{i-1},Y_{i-1}$.
    Let $X_i$ and $Y_i$ be (very small) homothetic $n$-gons inside $B_{\varepsilon_{i-1}}[q]$ such that the bottom corner of $X_i$ and the top corner of $Y_i$ coincide in a single point $p$ on $l - q$, as shown in \Cref{fig:CrossingUV8}.
    Note that $X_i$ and $Y_i$ have strictly more than $\norm{\P_n^{n/2}}$ of their area covered by $S_u$, i.e., strictly less then $\norm{\P_n^{n/2+2}}$ is free.
    Moreover, the closer $p$ is to $q$, the closer is their free fraction is to $\norm{\P_n^{n/2}+2} = s(n)$.
    Now we move $X_i,Y_i$ along $l$ until at least $\alpha < s(n)$ of each of their areas is not covered by $S_u$ and pick $\varepsilon_i > 0$ small enough so that $B_{\varepsilon_i}[q]$ is disjoint from $X_i$ and $Y_i$.
    Observe that our invariants hold again.
    
    After step $k$, we can place a (tiny) homothetic $n$-gon $S_v$ inside the ball $B_{\varepsilon_k}[q]$ such that its left corner coincides with $q$, which completes the construction for $n \equiv 0 \pmod 4$.

    \begin{figure}
        \centering
        \includegraphics{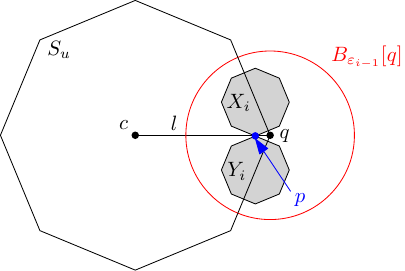}
        \caption{Placing $X_i, Y_i$ inside the $\varepsilon_{i-1}$ ball around $q$.}
        \label{fig:CrossingUV8}
    \end{figure}


    \begin{figure}
        \centering
        \includegraphics[page=2]{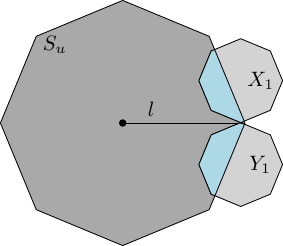}
        \hspace{2em}
        \includegraphics[page=2]{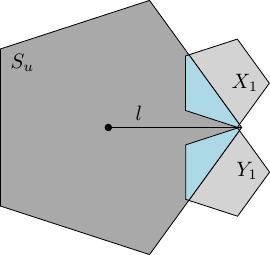}

        \vspace{-0.5ex}
        \includegraphics[page=2]{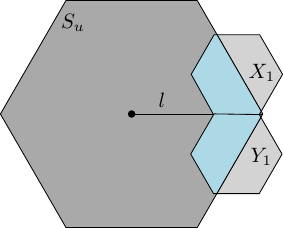}
        \hspace{2em}
        \includegraphics[page=2]{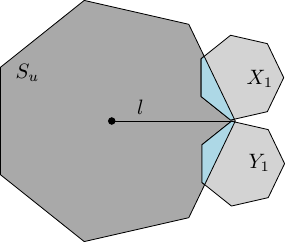}
        \caption{Placement of $X_1,Y_1$ in the cases $ n \equiv 0, 1, 2, 3 \mod 4 $}
        \label{fig:CrossingMod}
    \end{figure}

    If $n \not\equiv 1 \pmod 4$, then the regular $n$-gon $\P_n$ does not have a corner in each of its four extreme directions.
    Instead, we may have corners or entire sides.
    Nevertheless, we can assume that $\P_n$ has a right corner, and do the same construction as above for any $n$.
    However, depending on $n$, we get a different number of corners of $X_i$ and $Y_i$ inside $S_u$.
    For $n \equiv 1 \pmod  4$ the number of corners of $X_i$ inside of $S_u$ is $\lceil n/2 \rceil$ and thus roughly $\norm{\P_n^{\lceil n/2 \rceil +1}}$ of the area of $X_i$ is not covered by $S_u$.
    For $n \equiv 2 \pmod  4$ the number of corners is $n/2+1$ and roughly $\norm{\P_n^{n/2+1}} = \frac12$ is free.
    Lastly, for $n \equiv 3 \pmod  4$ the number of corners of $X_i$ inside of $S_u$ is $\lfloor n/2 \rfloor$ and thus roughly $\norm{\P_n^{\lceil n/2 \rceil +2}}$ of the area of $X_i$ is not covered by $S_u$.
    An example of the placement of $X_1,Y_1$ in the various cases is given in \Cref{fig:CrossingMod}.
%
\end{proof}

\subsection{\texorpdfstring{$\bm{k}$}{k}-independent crossing for \texorpdfstring{$\bm{\alpha \geq s(n)}$}{alpha >= s(n)}}
\label{sec:crossingKVertexCover}

We show that for every $n \geq 3$,  $\alpha \geq s(n)$, and for every collection of $\alpha$-free homothetic regular $n$-gons, the canoncial drawing $\Gamma$ of the corresponding intersection graph $G$ is $k$-independent crossing (cf.~\cref{def:k-independent-crossing}) for a global constant $k$ that depends only on $n$.

The reader might recall \cref{lem:3touching-overlap} stating that for large enough $\alpha$, no three $\alpha$-free $n$-gons have a common point.
We use a similar strategy to show that for any large enough $\alpha$ (in this case $\alpha \geq \frac12$), the number of $n$-gons containing a given point~$p$ is bounded.

\begin{restatable}{lemma}{manySharedPoint}
    \label{lem:bounded-clique-number}
    For a point $p \in \mathbb{R}^2$, $n\geq 5$, and $\alpha \geq \frac12$, there are at most $13$ regular $n$-gons $X \in \mathcal{C}$ with $p \in X$.
\end{restatable}


\begin{proof}
    We prove a slightly stronger statement, namely that if there is a point $p$ that is contained in the circumcircle of 14 $n$-gons, then at least one of the $n$-gons has more than half of its area covered by others.
    We actually prove that for one of the 14 regular $n$-gons, $S$,
    the covered area of the incircle of $S$ is at least $\frac{\norm{S}}{2}$.
    That is, even assuming that the only intersection of $S$ with other shapes is in its incircle, more than half of the area of $S$ is covered, showing that $S$ is not $\alpha$-free.
    We compute that, if $S$ has circumradius~$1$, then $\frac{\norm{S}}{2} = n \cdot \sin(\frac{2\pi}{n})/4 \leq 1.2$, for $n \geq 5$.

    First observe that the shared area of 14 circles is minimized by same-size circles with the common point $p$ on the boundary, where the centers are distributed with angles $\frac{2\pi}{14}$ around~$p$.
    Similarly, the shared area of the incircles of 14 regular $n$-gons is minimized by shapes of the same size (say with circumradius~$1$) with $p$ on the boundary of the circumcircles and the centers evenly distributed around $ p $.
    That is, we aim to show that in such a configuration of shapes with circumradius~1, the shared area in each incircle is at least 1.2.

    \begin{figure}
        \centering
        \includegraphics[page=1]{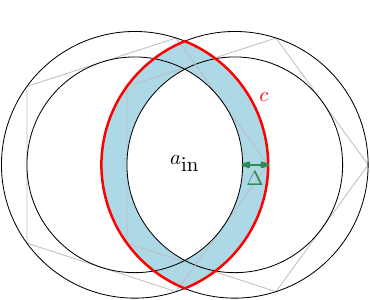}
        \hspace{2em}
        \includegraphics[page=2]{shared_point}
        \caption{
            Incircle and circumcircle of two regular 5-gons.
            The blue area is at most $ c \cdot \Delta $, i.e., we have $ a_{\text{in}} \geq a_{\text{circ}} - c \cdot \Delta $ (left).
            The angles between the centers is $2\pi/14$ and thus $ \theta = 12\pi/14$ (right).
        }
        \label{fig:shared_point}
    \end{figure}

    Let $ a_{\text{in}} $ and $ a_{\text{circ}} $ denote the area shared by two incircles, respectively two circumcircles, and let $ c_{\text{circ}} $ denote the circumference of the latter.
    Further let $ r = \cos(\frac{\pi}{n}) $ denote the inradius and let $ \Delta = 1 - r $ denote the difference between the inradius and the circumradius.
    See \cref{fig:shared_point} for an illustration of the notation.
    Now the shared area $ a_{\text{in}} $ of the incircles can be bounded by $ a_{\text{in}} \geq a_{\text{circ}} - c \cdot \Delta $.
    Using well-known formulas for area $ (\theta - \sin(\theta))/2 $ and arc length $ \theta $ of a unit disk segment with angle $ \theta $, we obtain the desired bound on $ a_{\text{in}} $ as follows (cf.~\cref{fig:shared_point}).
    First, the angle between the two intersection points of the circumcircle is $ \theta = \pi - 2 \cdot 2 \pi/(14 \cdot 2) = 12\pi/14 $.
    Next, the circumference $ c $ is twice the arc length, i.e., $ c = 2 \theta $,
    the shared area of two circumcircles is twice the area of the disk segments, i.e., $ a_{\text{circ}} = \theta - \sin(\theta) $,
    and for the difference of the radii we have $ \Delta = 1 - r = 1 - \cos(\pi/n) $.
    Together, we obtain 
    $ a_{\text{in}} \geq a_{\text{circ}} - c \cdot \Delta %
    = \theta - \sin(\theta) - 2 \theta \cdot (1 - \cos(\pi/n)) %
    \geq 1.2 $ for $ n \geq 5 $, as desired.
\end{proof}

The next lemma crucially exploits that $ s(n) $ is chosen as the tipping point whether or not two regular $n$-gons can meet in a third regular $n$-gon without containing a corner of the latter.

\begin{lemma}
    \label{lem:must-hit-p-or-corner}
    Let $X$ be a homothetic copy of $\P_n$.
    Let $c$ be the center, $Q$ be the set of corners, and $p$ be any fixed point on the boundary of $X$.
    Further, let $Y$ be another homothetic copy of $\P_n$ such that $Y \cap \overline{cp} \neq \emptyset$.
    Then at least one of the following holds.
    \begin{itemize}
        \item $\frac{\norm{Y - X}}{\norm{Y}} < s(n)$, i.e., less than an $s(n)$-fraction of $Y$ is not covered by $X$, or
        \item $Y \cap (Q \cup \{p\}) \neq \emptyset$, i.e., $Y$ contains point $p$ or a corner of $X$.
    \end{itemize}
\end{lemma}

\begin{proof}
    As we are done otherwise, we assume that $Y$ contains no corner of $X$ and that $Y$ is not completely contained in $X$.
    Then $Y$ intersects exactly one side $s$ of $X$.
    Let $h$ be the halfplane supported by $s$ that contains $X$.
    Then $Y - X = Y - h$, and hence $\norm{Y - X} / \norm{Y} \leq \norm{\P_n^{m}}$ where $m$ is the number of corners of $Y$ that lie outside of $X$ or on the boundary of $X$.
    For convenience let
    \begin{equation*}
        m^* =
            \begin{cases*}
                n/2+2                      & if $n \equiv 0 \pmod  4$ \\
                \lceil n/2 \rceil+1      & if $n \equiv 1 \pmod  4$ \\
                n/2+1              & if $n \equiv 2 \pmod  4$ \\
                \lceil n/2 \rceil+2    & if $n \equiv 3 \pmod  4$.
            \end{cases*}
    \end{equation*}
    I.e., $s(n) = \norm{\P_n^{m^*}}$, and we are done if $m < m^*$.
    So assume that $Y$ has $m \geq m^*$ corners outside or on the boundary of $X$, and note that this implies that the center of $Y$ lies outside or on the boundary of $X$.
    Also note that the radius of $ Y $ is smaller than the side length of $ X $.
    Together, it follows that the point $p$ either lies on the side $s$ of $X$ that is intersected by $Y$ or on an adjacent side $s'$ of $X$.
    Also we assume that $p \notin Y$, as otherwise we are done.

    Next, we shall argue that we may assume that $p$ is a corner of $X$.
    We actually only need $ p $ to satisfy the condition that $Y$ contains neither a corner of $ X $ nor $ p $ but intersects a side of $ X $ adjacent to $ p $ and the line segment $\overline{cp}$.
    In case that $p$ lies on a side $s' \neq s$ adjacent to $s$, then the corner $p'$ of $X$ where $s$ and $s'$ meet fulfills the same condition.
    In the other case that $p$ lies on the side $s$ that $Y$ intersects, we can simultaneously move $p$ and $Y$ parallel to $s$ until $p$ coincides with the corner $p'$ of $s$ on the far end of $Y$ (recall that $p \notin Y$).
    Seen from $Y$, this only changes the angle of the line segment $\overline{cp}$, making it lean even more towards $Y$.
    Again, the new situation fulfills the same condition.

    Finally, we face the situation of $p$ being a corner of $X$, while $Y$ contains no corner of $X$ but intersects a side $s$ of $X$ incident to $p$ and intersects the line segment $\overline{cp}$.
    Then the portion of $\norm{Y}$ outside $X$ is, by definition, strictly less than $s(n)$, which concludes the proof.
\end{proof}

With \cref{lem:must-hit-p-or-corner,lem:bounded-clique-number} at hand, we can now prove the following.

\begin{proposition}
    \label{prop:alpha-large-independent}
    Let $n \geq 3$ and $\alpha \geq s(n)$ be fixed.
    Then for any collection $\mathcal{C} = \{S_v\}_{v \in V}$ of $\alpha$-free homothetic regular $n$-gons with intersection graph $G = (V,E)$, the canonical drawing $\Gamma$ of $G$ is $26(n+1)$-independent crossing.
\end{proposition}

\begin{proof}
    For $n \in \{3,4\}$, we have $s(n) = 1$.
    Hence, $\Gamma$ is planar for $n=3$ and $1$-planar for $n=4$, which is more than enough.
    For $n \geq 5$, let $uv \in E$ be any fixed edge in $G$.
    Our task is to bound the number of independent edges in $G$ that cross $uv$ in~$\Gamma$.

    Recall from \cref{obs:drawing}, that $S'_u \subseteq S_u$ and $S'_v \subseteq S_v$ are interiorly disjoint but touching homothetic $n$-gons, and edge $uv$ is drawn as a polyline from $u$ ($\varepsilon$-close to the center of $S_u$) to $v$ ($\varepsilon$-close to the center of $S_v$) with one bend point $\varepsilon$-close to $p_{uv} \in S'_u \cap S'_v$.
    Actually, we may assume without loss of generality that $u$ is the center of $S'_u$, $v$ is the center of $S'_v$, and edge $uv$ bends exactly at $p_{uv}$.
    Then $uv$ consists of two line segments $\overline{up_{uv}}$ and $\overline{p_{uv}v}$, and it is enough to bound the number of edges that cross one of these line segments in~$\Gamma$, say $\overline{up_{uv}}$.

    For every edge $xy$ that crosses $\overline{up_{uv}}$ in $\Gamma$, we have $S_x \cap \overline{up_{uv}} \neq \emptyset$ or $S_y \cap \overline{up_{uv}} \neq \emptyset$, or both.
    Let $A = \{y \in V \mid S_y \cap \overline{up_{uv}} \neq \emptyset\}$ be the subset of vertices of $G$ whose corresponding sets intersect $\overline{up_{uv}}$.
    Crucially, all edges of $G$ that cross $uv$ in $\Gamma$ along the line segment $\overline{up_{uv}}$ have an endpoint in $A$.
    Since $S'_u \subseteq S_u$, we have 
    \[
        \frac{\norm{S_y - S'_u}}{\norm{S_y}} \geq \frac{\norm{S_y - S_u}}{\norm{S_y}} \geq \alpha \geq s(n)
    \]
    for each $y \in A$.
    Hence, by \cref{lem:must-hit-p-or-corner}, each such $S_y$ must contain the point $p_{uv}$ or (at least) one of the $n$ corners of $S'_u$.
    \cref{lem:bounded-clique-number} says that at most~$13$ such $S_y$ can contain the same point, and thus $|A| \leq 13 \cdot (n+1)$.
    In other words, at most $13(n+1)$ independent edges of $G$ cross $uv$ in $\Gamma$ along the line segment $\overline{up_{uv}}$.
    Symmetrically, at most $13(n+1)$ independent edges cross $\overline{p_{uv}v}$, and thus $\Gamma$ is $k$-independent crossing for $k = 26(n+1)$.
\end{proof}

Finally, \cref{prop:alpha-large-independent,prop:alpha-small-no-independent} together prove \cref{thm:results_k-independent-crossing}.

\resultskindependentcrossing*

\section{Conclusion}
\label{sec:conclusion}

It remains an intriguing problem to determine for the regular $n$-gon $\P_n$ the threshold $\alpha^*(\P_n)$ such that the class $\mathcal{G}(\P_n,\alpha)$ of intersection graphs of $\alpha$-free homothetic copies of $\P_n$ admits product structure for $\alpha > \alpha^*(\P_n)$ and no product structure for $\alpha < \alpha^*(\P_n)$.

With $s(n)$, as defined in~\eqref{eq:threshold-definition}, we determined the exact threshold for $\alpha$ such that for $\alpha < s(n)$ arbitrarily many independent edges can cross a single edge in canonical drawings.
While this is exactly the crucial ingredient, we can not construct the nested grids for $\mathcal{G}(\P_n,\alpha)$ with $\alpha \approx s(n)$, unless $n \in \{4,6\}$.
Still, we suspect an alternative construction to work.

\begin{conjecture}\label{con:NoProductStructure}
    For every $\alpha < s(n)$, the class of intersection graphs of $\alpha$-free homothetic regular $n$-gons does not have product structure.
\end{conjecture}

On the other hand, for $\alpha \geq s(n)$, the canonical drawings of $\mathcal{G}(\P_n,\alpha)$ are $k$-independent crossing.
We prove this for $k = 26(n+1)$ (cf.\ \cref{prop:alpha-large-independent}) but suspect that a constant $k$ independent of $n$ should suffice.
As already conjectured in the introduction, we believe that graph classes with $k$-independent crossing drawings have product structure.

\begin{conjecture}
    The class of $k$-independent crossing graphs admits product structure.
\end{conjecture}

By \cref{thm:results_k-independent-crossing} this would imply the following conjecture matching \cref{con:NoProductStructure}.

\begin{conjecture}\label{con:HaveProductStructure}
    For every $\alpha \geq s(n)$, the class of intersection graphs of $\alpha$-free homothetic regular $n$-gons has product structure.
\end{conjecture}

This seems reasonable since similar beyond-planar graph classes, such as $k$-planar graphs\cite{ProductStructurekPlanarGraphs}, fan-planar graphs and $k$-fan-bundle graphs \cite{ShallowMinorsUSW}, have been shown to have product structure.
In particular, Hickingbotham and Wood~\cite{ShallowMinorsUSW} show that if all graphs in $\mathcal{G}$ are $r$-shallow minors of $H \boxtimes K_l$ with $r,l, \tw(H) \in O(1)$, then $\mathcal{G}$ has product structure.
For example, they show this to be true for fan-planar graphs. 


\bibliography{bibliography}    

\end{document}